\documentclass[12pt,a4paper]{amsart}
\usepackage{pdfsync}

\normalfont
\usepackage{enumerate}

\usepackage[applemac]{inputenc}
\usepackage[T1]{fontenc}

\usepackage{amssymb,amsrefs,amscd}
\usepackage{marvosym}
\usepackage{graphicx}
\usepackage{epstopdf}
\usepackage{lscape}
\usepackage{color}
\usepackage{verbatim}

\usepackage[curve,tips]{xypic}
\SelectTips{eu}{11}
\UseTips

\renewcommand{\epsilon}{\varepsilon}

\renewcommand{\emptyset}{\varnothing}

\newtheorem{theorem}{Theorem}[section]

\newtheorem{corollary}[theorem]{Corollary}
\newtheorem{lemma}[theorem]{Lemma}

\newtheorem{question}[theorem]{Question}

\newtheorem*{Main Theorem}{Main Theorem}

\theoremstyle{definition}

\newtheorem{definition}[theorem]{Definition}

\theoremstyle{remark}

\DeclareMathOperator{\cd}{cd}
\DeclareMathOperator{\gd}{gd}

\newcommand{\mF}{\mathfrak {F}}

\newcommand{\FF}{\mathfrak F}
\newcommand{\GG}{\mathfrak G}

\newcommand{\Z}{\mathbb Z}

\newcommand{\R}{\mathbb R}

\newcommand{\cohom}[3]{H^{{\raise1pt\hbox{$\scriptstyle#1$}}}(#2\>\!,#3)}
\newcommand{\tatecohom}[3]%
  {\widehat H^{{\raise1pt\hbox{$\scriptstyle#1$}}}(#2\>\!,#3)}

\newcommand{\Cohom}[3]%
  {H^{{\raise1pt\hbox{$\scriptstyle#1$}}}\big(#2\>\!,#3\big)}
\newcommand{\Tatecohom}[3]%
  {\widehat H^{{\raise1pt\hbox{$\scriptstyle#1$}}}\big(#2\>\!,#3\big)}

\newcommand{\homol}[3]{H_{{\lower1pt\hbox{$\scriptstyle#1$}}}(#2\>\!,#3)}
\newcommand{\homolog}[2]{H_{{\lower1pt\hbox{$\scriptstyle#1$}}}(#2)}

\newcommand{\epi}{\twoheadrightarrow}

\newcommand{\eg}{{\underline EG}}

\newcommand{\uueg}{{\underline{\underline E}}G}
\newcommand{\EFG}{E_{\frak F}G}

\newcommand{\bs}{\backslash}

\catcode`\@=11

\newcommand{\calO}{\mathcal O}
\newcommand{\ModOFG}{\mathop{{\operator@font
Mod\text{-}}\calO_{\frakF}G}}
\newcommand{\OFGMod}{\mathop{\calO_{\frakF}G\text{-}{\operator@font
Mod}}}
\newcommand{\ModOGG}{\mathop{{\operator@font
Mod\text{-}}\calO_{\frakG}G}}
\newcommand{\OGGMod}{\mathop{\calO_{\frakG}G\text{-}{\operator@font
Mod}}}

\newcommand{\All}{\mathfrak{A}ll}

\newcommand{\Ffin}{\frakF_{\operator@font fin}}
\newcommand{\Fvc}{\frakF_{\operator@font vc}}
\newcommand{\Fic}{\frakF_{\operator@font ic}}
\newcommand{\Ffg}{\frakF_{\operator@font fg}}
\newcommand{\Fpc}{\frakF_{\operator@font pc}}
\newcommand{\Fab}{\frakF_{\operator@font ab}}
\newcommand{\Fvpc}{\frakF_{\operator@font vpc}}
\newcommand{\Fvab}{\frakF_{\operator@font vab}}

\catcode`\@=12

\renewcommand{\coprod}%
{\mathop{\rotatebox[origin=c]{180}{$\displaystyle\prod$}}\limits}

\title[Classifying spaces for families of subgroups]{On the dimension of classifying spaces for families of abelian subgroups}

\author[G. Corob Cook]{Ged Corob Cook}
\address{Mathematics, University of Southampton, 
Southampton, SO17 1BJ, UK}
\email{ged.corobcook@cantab.net}

\author[V. Moreno]{Victor Moreno}
\address{ Department of Mathematics,
Royal Holloway, University of London, 
Egham, TW20 0EX, UK }
\email{victormglezmoreno@gmail.com}

\author[B. Nucinkis]{Brita Nucinkis}
\address{ Department of Mathematics,
Royal Holloway, University of London, 
Egham, TW20 0EX, UK }
\email{Brita.Nucinkis@rhul.ac.uk}

\author[F. Pasini]{Federico Pasini}
\address{Department of Mathematics and Applications
University of Milano-Bicocca
Via Cozzi, 55
20125 Milano, ITALY}
\email{f.pasini1@campus.unimib.it}

\thanks{The fourth author gratefully acknowledges the support by the Italian \textit{National Group for Algebraic and Geometric Structures and Their Applications} (GNSAGA -- INDAM)}

\date{\today}

\begin{document}

\begin{abstract} We show that a finitely generated abelian group $G$ of torsion-free rank $n\geq 1$ admits a $n+r$ dimensional model for  $E_{\FF_r} G$, where $\mF_r$ is the family of subgroups of torsion-free rank less than or equal to $r\geq 0$. 

\end{abstract}

\maketitle

\thispagestyle{empty}


\section{Introduction}

In this note we consider classifying spaces $\EFG$ for a family of subgroups $\mF$ of $G$. We are particularly interested in the minimal dimension, denoted  $\gd_\mF G,$ such a space can have. 

Let $G$ be a group. We say a collection of subgroups $\mF$ is a family if it is closed under conjugation and taking subgroups. A $G$-CW-complex $X$ is said to be a classifying space $\EFG$ for the family $\mF$ if, for each subgroup $H \leq G$, $X^H \simeq \{\ast\}$ if $H \in \mF$, and $X^H =\emptyset$ otherwise.

The spaces $\eg= \EFG$ for $\mF = {\mF}in$ the family of finite subgroups and $\uueg =\EFG$ for $\mF=\mathcal{V}cyc$ the family of virtually cyclic subgroups have been widely studied for their connection with the Baum-Connes and Farrell-Jones conjectures respectively. For a first introduction into the subject see, for example, the survey \cite{luecksurvey}.

We consider finitely generated abelian groups $G$ of finite torsion-free rank $r_0(G)=n$ and families $\FF_r$  of subgroups of torsion-free rank less or equal to $r<n.$
Note that for $r=0,$  $\FF_0={\FF}in$ and that it is a well known fact, see for example \cite{luecksurvey}, that $\R^n$ is a model for $\eg$ and that $\gd_{\FF_0}G=n.$ 
For $r=1,$ $\FF_1={\mathcal{V}}cyc$ and it was shown in \cite[Proposition 5.13(iii)]{lueckweiermann} that $\gd_{\FF_1}G=n+1.$

The main idea is to use the method developed by L\"uck and Weiermann \cite{lueckweiermann} to build models of $E_{\FF_r}G$ from models for $E_{\FF_{r-1}}G.$
We begin by recalling those results in \cite{lueckweiermann} that we need for our construction.  
Let $\FF$ and $\GG$ families of subgroups of a given group $G$ such that $\FF \subseteq \GG$.

\begin{definition}\label{equiv}\cite[(2.1)]{lueckweiermann}
Let $\FF$ and $\GG$ be families of subgroups of a given group $G$ such that $\FF \subseteq \GG$. Let $\sim$ be an equivalence relation on $\GG\backslash\FF$ satisfying:
\begin{itemize} 
\item For $H,K \in \GG\backslash\FF$ with $H \leq K$ we have $H \sim K.$
\item Let $H,K \in \GG\bs\FF$ and $g \in G$, then $H \sim K \iff gHg^{-1} \sim gKg^{-1}.$
\end{itemize}
Such a relation is called a strong equivalence relation.
Denote by $[\GG\bs\FF] $ the equivalence classes of $\sim$ and define for all $[H]\in [\GG\bs\FF]$ the following subgroup of $G$:
$$N_G[H] =\{g \in G\,|\, [gHg^{-1}]=[H]\}.$$
Now define a family of subgroups of $N_G[H]$ by
$$\GG[H] =\{K \leq N_G[H] \,|\, K \in \GG\bs\FF \, ,\, [K]=[H]\}\cup (\FF \cap N_G[H]).$$
Here $\FF \cap N_G[H]$ is the family of subgroups of of $N_G[H]$ belonging to $\FF$.
\end{definition}

\begin{theorem}\label{lw-main}\cite[Theorem 2.3]{lueckweiermann} Let $\FF \subseteq \GG$ and $\sim$ be as in Definition \ref{equiv}. Denote by $I$ a complete set of representatives of the conjugacy classes in $[\GG\bs\FF].$ Then the $G$-CW-complex given by the cellular $G$ push-out
$$\xymatrix{\sqcup_{[H]\in I} G \times_{N_G[H]}E_{\FF \cap N_G[H]}(N_G[H])  \ar[r]^{\phantom{ppppppppppppp}i} \ar[d]_{\sqcup_{[H]\in I} id_G \times_{N_G[H]}f_{[H]}}  & E_{\FF}(G) \ar[d] \\
\sqcup_{[H]\in I} G \times_{N_G[H]}E_{\GG[H]}(N_G[H]) \ar[r]  &X}$$
where either $i$ or the $f_{[H]}$ are inclusions, is a model for $E_\GG(G).$
\end{theorem}

The condition on the two maps being inclusions is not that strong a restriction, as one can replace the spaces by the mapping cylinders, see \cite[Remark 2.5]{lueckweiermann}. Hence one has:

\begin{corollary}\label{lw-dim}\cite[Remark 2.5]{lueckweiermann}
Suppose there exists an $n$-dimensional model for $E_{\FF}G$ and, for each $H \in I$, a $(n-1)$-dimensional model for $E_{\FF \cap N_G[H]}(N_G[H])$ and a $n$-dimensional model for $E_{\GG[H]}(N_G[H])$. Then there is an $n$-dimensional model for $E_\GG G.$
\end{corollary}

Corollary \ref{lw-dim} gives us  a tool to find an upper bound for $\gd_\GG G$.  A very useful tool to find a lower bound for $\gd_\GG G$ is the following Mayer-Vietoris sequence \cite{lueckalgtop}, which is an immediate consequence of Theorem \ref{lw-main}, see also \cite[Proposition 7.1]{degrijsepetrosyan} for the Bredon-cohomology version.

\begin{corollary}\label{MV}  With the notation as in Theorem \ref{lw-main} we have following long exact cohomology sequence:
$$... \to H^i(G\backslash E_\GG G) \to (\prod_{[H]\in I} H^i(N_G[H]\backslash E_{\GG[H]}N_G[H]) )\oplus  H^i(G\backslash E_\FF G) \to $$
$$\prod_{[H]\in I} H^i(N_G[H]\backslash E_{\FF\cap N_G[H]}N_G[H])
 \to H^{i+1} (G\backslash E_\GG G) \to ...$$

\end{corollary}

This note will be devoted to proving the following Theorem:

\begin{Main Theorem}\label{main} Let $G$ be a finitely generated abelian group of finite torsion-free rank $n\geq 1$, and denote by $\FF_r$ the family of subgroups of torsion-free rank less than or equal to $r\geq 0$.
Then
$$ \gd_{\FF_r} G \leq n+r.$$
\end{Main Theorem}

The case for more general classes of groups $G$ is going to be dealt with, using different methods, by the second author in his Ph.D thesis.

\section{The Construction}

Throughout, let $G$ denote a finitely generated abelian group of torsion-free rank $r_0(G)=n$.

The idea is to construct models for $E_{\FF_r} G$ in terms of models for $E_{\FF_{r-1}} G$  using the push-out of Theorem \ref{lw-main} inductively.
As a first step we shall define an equivalence relation in the sense of Definition \ref{equiv}.

\begin{lemma}
Let $\sim$ denote the following relation on $\FF_r \backslash \FF_{r-1}:$ 
$$H \sim K \iff rk (H\cap K)=r.$$
Then
 $\sim$ is a strong equivalence relation.
 \end{lemma}

\begin{proof} We show that $\sim$ is transitive: If $H \sim K$ and $K \sim L$, this implies that both $H\cap K$ and $K\cap L$ are finite index subgroups of $K$. Hence also $H\cap K \cap L$ is a finite index subgroup of $K$, and in particular of $K\cap L$ and thus of $L$. Hence $H\cap L$ is finite index in both $H$ and $L$.
The rest is easily checked.
\end{proof}

\begin{lemma}\label{unimax} $G$ satisfies $(M_{\FF_{r-1} \subseteq \FF_r})$, i.e. every subgroup $H \in \FF_{r}\bs\FF_{r-1}$ is contained in a unique $H_{max} \in \FF_{r}\bs\FF_{r-1}$, which is maximal.
\end{lemma}

\begin{proof} The existence follows from \cite{segal}. As regards uniqueness, suppose $H$ is included in two different maximal elements $K,L\in\FF_{r}\bs\FF_{r-1}$: then $H\leq KL$.
Note that, since $H \sim L$ and $H \leq L$, it follows that $|L\colon H|<\infty$.
Hence 
$$|KL\colon K|=|L\colon K\cap L|\leq |L\colon H|<\infty$$
implies $KL\in\FF_{r}\bs\FF_{r-1}$, contrary to the maximality of $K$ and $L$.
\end{proof}

Note that we always have maximal elements in $\FF_{r}\bs\FF_{r-1}$ as long as the ambient group is polycyclic \cite{segal}, but uniqueness already fails for the Klein-bottle group $K$, which is non-abelian but contains a free abelian subgroup of rank $2$ as an index $2$ subgroup. Denote 
$$K=\langle a,b \,|\, aba^{-1}=b^{-1}\rangle$$
and consider $\FF_1$ the family of cyclic subgroups. Since $a^2=(ab^{-1})^2$, in follows that $\langle a^2 \rangle \leq \langle ab^{-1} \rangle$ as well as $\langle a^2 \rangle \leq \langle a \rangle$, both of which are maximal.

For $M \leq G$ a subgroup of $G$ we denote by ${\All(M)}$ the family of all subgroups of $M$.

\begin{lemma}\label{eallm}
Let $M$ be a maximal subgroup of $G$ of torsion-free rank $r$. Then $\R^{n-r}$ is a model for $E_{\All(M)}G,$ and $\gd_{\All(M)}G = n-r.$
\end{lemma}

\begin{proof} Since $M$ is maximal it follows that $G/M$ is torsion-free of rank $n-r$ and hence $\R^{n-r}$ is a model for $E(G/M).$ The action of $G$ given by the projection $G \epi G/M$ now yields the claim.
\end{proof}

\begin{lemma}\label{federicovictor}
Let $\FF$ and $\GG$ be two families of subgroups of $G$. Then 
$$\gd_{\FF\cup\GG}G \leq \max\{\gd_\FF G, \gd_\GG G, \gd_{\FF\cap\GG}G +1\}.$$
\end{lemma}

\begin{proof} 
By the universal property of classifying spaces for families, there are maps, unique up to $G$-homotopy, $E_{\FF\cap\GG}G \to E_{\GG}G$ and $E_{\FF\cap\GG}G \to E_\FF G.$ Now the double mapping cylinder yields a model for $E_{\FF\cup \GG}G$ of the desired dimension.
\end{proof}

\begin{lemma}\label{secondpushout}
Given $r <n$, suppose there exists a $d \geq n$ such that $\gd_{\FF_{r-1}} G \leq d$ and that for all maximal subgroups $N$ with $r_0(N) >r-1$ we also have $\gd_{\FF_{r-1}\cap\All(N)} G \leq d.$ Then
$$\gd_{\FF_r}G \leq d+1 \quad \mbox{ and } \quad \gd_{\FF_r\cap\All(M)}G \leq d+1$$
for all maximal subgroups $M$ of $r_0(M)> r.$
\end{lemma}

\begin{proof}
We begin by applying Theorem \ref{lw-main} to the families $\GG=\FF_r$ and $\FF=\FF_{r-1}.$
Lemma \ref{unimax} implies that $G$ satisfies $(M_{\FF_{r-1} \subseteq \FF_r}).$
Denote by $\mathcal{N}$ the set of equivalence classes of  maximal elements in $\FF_r\backslash \FF_{r-1}$. Then
 \cite[Corollary 2.8]{lueckweiermann} gives   a push-out:
 
$$\xymatrix{\sqcup_{N \in \mathcal{N}} E_{\FF_{r-1}}(G)  \ar[r] \ar[d]  & E_{\FF_{r-1}}(G) \ar[d] \\
\sqcup_{N \in \mathcal{N}} E_{\FF_{r-1} \cup\All(N)}(G) \ar[r]  &Y,}$$
and $Y$ is a model for $E_{\FF_r}G.$ 

By assumption we have that $\gd_{\FF_{r-1}}G \leq d$ and $\gd_{\FF_{r-1}\cap\All(N)}G\leq d$ for all $N \in \mathcal{N}.$ Furthermore, by Lemma \ref{eallm} we have that $\gd_{\All(N)}G = n-r_0(N) < n.$ Lemma \ref{federicovictor} now implies that $\gd_{\FF_{r-1}\cup \All(N)}G\leq d+1$.
Applying  Corollary \ref{lw-dim}  to the above push-out yields
$$\gd_{\FF_r}G\leq d+1.$$

The second claim is proved similarly applying Theorem \ref{lw-main} to the families $\GG={\FF_r\cap\All(M)}$ and $\FF={\FF_{r-1}\cap\All(M)}.$
The argument of Lemma \ref{unimax} applies here as well and hence $G$ satisfies $(M_{(\FF_{r-1}\cap\All(M)) \subseteq (\FF_r\cap\All(M))}).$
We denote by $\mathcal{N}(M)$ the set of equivalence classes of
maximal elements in $\FF_r\cap\All(M)\backslash\FF_{r-1}\cap\All(M).$ this now gives us a push-out:
$$\xymatrix{\sqcup_{N \in \mathcal{N}(M)} E_{\FF_{r-1}\cap\All(M)}(G)  \ar[r] \ar[d]  & E_{\FF_{r-1}\cap\All(M)}(G) \ar[d] \\
\sqcup_{N \in \mathcal{N}(M)} E_{(\FF_{r-1}\cap\All(M)) \cup\All(N)}(G) \ar[r]  &Z,}$$
and $Z$ is a model for $E_{\FF_r\cap\All(M)}G.$ 

Since $N\leq M$, it follows that $(\FF_{r-1}\cap\All(M)) \cap\All(N) = \FF_{r-1}\cap\All(N)$ and hence, by assumption $\gd_{(\FF_{r-1}\cap\All(M)) \cap\All(N)} G \leq d$ and Lemma \ref{federicovictor} implies that $\gd_{(\FF_{r-1}\cap\All(M)) \cup\All(N)} G \leq d+1.$ Now the same argument as above applies and
$$\gd_{\FF_r\cap\All(M)}G \leq d+1.$$

\end{proof}

\bigskip\noindent{\bf Proof of Main Theorem:} 
We begin by noting that for $r=0$ we have that $\FF_r=\FF_0$ is the family of all finite subgroups of $G$. Then for all 
maximal subgroups $M$ of rank $1$, we have that $\FF_0=\FF_0\cap\All(M).$ Furthermore, is is well known that $\gd_{\FF_0}G =n$, see for example  \cite{luecksurvey}.

Now an induction using Lemma \ref{secondpushout} yields the claim.

\qed

\begin{question} Is the bound of our Main Theorem sharp, i.e. for $n> r,$ is $$\gd_{\FF_r}G=n+r?$$
\end{question}

Since $\gd_{\FF_0}G = \gd_{\FF_0\cap\All(N)}G =n$ for all maximal subgroups $N$, we can assume equality in the inductive step (assumptions of Lemma \ref{secondpushout}). Then a successive application of the Mayer-Vietoris sequences to the push-outs in Lemmas \ref{secondpushout} and \ref{federicovictor}, reduces the question to whether the map
$$H^d(G\backslash E_{\FF_{r-1}}G) \to H^d(G\backslash E_{\FF_{r-1}\cap\All(N)}G)$$
is surjective or not.

\bigskip\noindent
We know by \cite{lueckweiermann} that $\gd_{\FF_1}G=n+1$ and it was shown in \cite{ana} that the question has a positive answer for $G=\Z^3,$ i.e. that $\gd_{\FF_2}(\Z^3)=5.$

\end{document}